\documentclass[review]{elsarticle}
\usepackage{lineno}
\usepackage[latin1]{inputenc}
\usepackage{times}
\usepackage{amssymb,mathrsfs, amsmath}
\usepackage{amsmath,amsthm}
\usepackage{amssymb}
\usepackage{latexsym}
\usepackage{amsfonts}
\usepackage{mathrsfs}
\usepackage{epsfig}
\usepackage{hyperref}
\usepackage{graphicx}
\usepackage{graphics}

\newtheorem{thm}{Theorem}[section]

\newtheorem{defn}{Definition}[section]
\newtheorem{prop}{Proposition}[section]

\newtheorem{lem}{Lemma}[section]

\newtheorem{rem}{Remark}[section]

\newtheorem{cor}{Corollary}[section]

\newtheorem{exmpl}{Example}[section]

\journal{nowhere}

\begin{document}

\begin{frontmatter}

\title{Equivariant one-parameter deformations of   associative algebra morphisms}

%
\author{RB Yadav}
\address{Sikkim University, Gangtok, Sikkim, 737102, \textsc{India}}
\ead{rbyadav15@gmail.com}
\begin{abstract}
In this article, we introduce equivariant formal deformation theory of associative algebra morphisms. We  introduce  an equivariant deformation cohomology of associative algebra morphisms  and using this we study the   equivariant formal deformation theory of associative algebra morphisms.

\end{abstract}

\begin{keyword}
\texttt{Group actions, Hochschild cohomology, formal deformations, equivariant cohomology}
\MSC[2010]   13D03 \sep 13D10 \sep 14D15 \sep 16E40 \sep  55M35 \sep 55N91
\end{keyword}

\end{frontmatter}



\section{Introduction}\label{rbsec1}
Origin of the idea of deformation theory goes back to a paper of Riemann on abelian functions published in 1857. Kodaira and Spencer initiated deformation theory  of complex analytic structures \cite{KS1}, \cite{KS2}, \cite{KS3}. M. Gerstenhaber introduced algebraic deformation theory in a series of papers \cite{MG1},\cite{MG2},\cite{MG3}, \cite{MG4}, \cite{MG5}. He studied deformation theory of associative algebras.  Deformation theory of associative algebra morphisms was introduced by M. Gerstenhaber and S.D. Schack \cite{GS1}, \cite{GS2}, \cite{GS3}. Nijenhuis and Richardson introduced deformation theory of Lie algebras  \cite{NR1}, \cite{NR2}. Deformation theory  of dialgebras has  been studied by Majumdar and Mukherjee  \cite{AG}. Recently, deformation theory of dialgebra morphisms and Leibniz algebra morphisms have been studied by Donald Yau and Ashis Mondal respectively \cite{DY}, \cite{AM}. Equivariant deformation theory of associative algebras has been studied in \cite{gmrb}.

Organization of the paper is as follows. In Section \ref{rbsec2}, we recall some definitions and results. In Section \ref{rbsec3}, we introduce equivariant deformation complex and equivariant deformation cohomology of an associative algebra morphism. In Section \ref{rbsec4}, we introduce equivariant deformation of an associative algebra morphism. In this section  we prove  that obstructions to equivariant deformations are cocycles. In Section \ref{rbsec5}, we study equivalence of two equivariant deformations and rigidity of an equivariant associative algebra morphism.

\section{Preliminaries}\label{rbsec2}
In this section, we recall definitions of associative algebra, associative algebra morphisms, Hochschild cohomology  and equivariant deformation cohomology of an associative algebra. Also, we recall definitions of a module over an associative algebra and module over an associative algebra morphism.  Throughout the paper we denote a fixed field  by k and a finite group by G.
\begin{defn}
  An associative algebra A is a k-module equipped with a k-bilinear map $\mu$ satisfying
  $$\mu(a,\mu(b, c))=\mu(\mu(a, b), c),$$ for all $a,b,c\in A$

  Let A be an associative k-algebra. A bimodule M over A is a k-module M with two actions (left and right) of A, $\mu:A\times M\to M$ and $\mu:M\times A\to M $ (for simplicity we denote both the actions by same symbol, one can differentiate both of them from the context) such that $\mu(x,\mu(y,z))=\mu(\mu(x,y),z),$ whenever one of x,y,z is from M and others are from A.

  Let A and B be associative k-algebras. An associative algebra morphism $\phi:A\to B$ is a k-linear map satisfying $$\phi(\mu(a, b))=\mu(\phi a,\phi b),$$ for all $a,b\in A.$

\end{defn}
\begin{exmpl}
  Let $M_n(\mathbb{R})$ be the collection of all $n\times n$ matrices with entries in $\mathbb{C}$. Then $M_n(\mathbb{C})$ is an associative algebra over $\mathbb{C}$ with respect to matrix multiplication.
\end{exmpl}

\begin{exmpl}
  Let $X$ be a nonempty set and  $A=\{\alpha : \alpha:X\to \mathbb{C}  \;\text{is a function}\}$. Then A is an associative algebra over $\mathbb{C}$ with the product $(\alpha\beta)(x)=\alpha(x)\beta(x).$
\end{exmpl}

\begin{exmpl}
Let V be a vector space over k. Define tensor module by  $$T(V)=k\oplus V\oplus \cdots\oplus V^{\otimes n}\oplus\cdots$$  $T(V)$ is an associative algebra with the concatenation product $T(V)\otimes T(V)\to T(V)$ given by $$v_1\cdots v_p\otimes v_{p+1}\cdots v_{p+q}=v_1\cdots v_p v_{p+1}\cdots v_{p+q}.$$
\end{exmpl}

\begin{defn}
  Let A be an associative k-algebra and M be a bimodule over A.  Let $C^n(A;M)=\hom_k(A^{\otimes n},M),$ for all integers $n\ge 0.$ Also, define a k-linear map $\delta^n:C^n(A;M)\to C^{n+1}(A;M)$ given by
  \begin{eqnarray*}
    \delta ^nf(x_1,\cdots, x_{n+1}) &=& x_1f(x_2,\cdots,x_{n+1})+\sum_{i=1}^{n}(-1)^if(x_1,\cdots,x_ix_{i+1}, \cdots, x_{n+1}) \\
     && +(-1)^{n+1}f(x_1,\cdots,x_n)x_{n+1}.
  \end{eqnarray*} This gives a cochain complex  $(C^{\ast}(A;M),\delta)$, cohomology of which is denoted by $H^{\ast}(A;M)$ and called as Hochschild cohomology of A with coeffiecients in M. A is a bimodule over itself in an obvious way. So we can consider Hochschild cohomology $H^{\ast}(A;A)$.\\
  Let A be an associative k-algebra with product $\mu (a, b) = ab$ and G be a finite group. The group G is said to act on A from the left if there exists a function
$$\phi:G\times  A \to A\;\;\;\; (g, a)\mapsto \phi(g, a) = ga$$
satisfying the following conditions.
\begin{enumerate}
  \item $ex = x$ for all $x \in A$, where $e \in G$ is the group identity.
  \item $g_1(g_2x) = (g_1g_2)x$ for all $g_1, g_2 \in G$ and $x \in A$.
  \item  For every $g \in G$, the left translation $\phi_g = \phi (g, ) : A \to A,$ $a \to ga$ is a linear
map.
  \item For all $g \in G$ and $a, b \in A$,  $\mu(ga, gb) = g\mu(a, b) = g(ab)$, that is, $\mu$ is equivari-
ant with respect to the diagonal action on $A \times A$.
\end{enumerate}

An action as above is denoted  by $(G,A)$ \cite{gmrb}.
Define $$C_G^n(A;M)=\{c\in C^n(A;M): c(gx_1,\cdots, gx_n)=gc(x_1,\cdots, x_n), \text{ for all}\; g\in G\}$$
An element in $C_G^n(A;M)$ is called an invariant n-cochain.  Clearly, $C_G^n(A;M)$ is a submodule of $C^n(A;M)$.
\end{defn}

From \cite{gmrb} we have following lemma.
\begin{lem}\label{rb1}
  If an n-cochain c is invariant then $\delta^n(c)$ is also invariant. In other words, $$c \in C_G^n(A;M) \Rightarrow \delta^n(c) \in C_G^{n+1}(A;M)$$
\end{lem}
From \cite{gmrb}, we have $(C_G^{\ast}(A;M), \delta)$  is a cochain complex. Cohomology of this complex is called equivariant deformation cohomology of A.
\begin{defn}
  Let $A$ and $B$ be associative k-algebras,  and $\phi :A\to B$ be an associative algebra morphism. Let $M$ and $N$ be  k-modules. A k-linear map $T:M\to N$ is said to be a left (or right) module over $\phi$ if following conditions are satisfied.
  \begin{enumerate}
    \item  $M$ and $N$ are left ( or right) modules over $A$ and $B$ respectively.
    \item T  is a left (or right) A-module morphism when N is viewed as a left (or right) A-module by virtue of the morphism $\phi :A\to B.$
  \end{enumerate}
   A k-linear map $T:M\to N$ is said to be a bimodule over $\phi$ if $T$ is a left as well as right module over $\phi$.
\end{defn}
From \cite{gmrb}, we recall equivariant deformation of an associative algebra  morphism.
\begin{defn}
Let A be an associative k-algebra with a $(G,A)$ action. Define $A_t=\{\sum_{i=0}^{\infty}a_it^i:a_i\in A\}.$  An equivariant formal one-parameter deformation of A is a k-bilinear multiplication $m_t:A_t\times A_t\to A_t$ satisfying the following properties:
\begin{enumerate}
  \item $m_t(a,b)=\sum_{i=0}^{\infty}m_i(a,b)t^i,$ for all $a,b\in A,$ where $m_i:A\times A\to A$ are k-bilinear and $m_0(a,b)=ab$ is the original multiplication on A, and $m_t$ is associative.
  \item For every $g\in G$, $m_i(ga,gb)=gm_i(a,b)$, for all $a,b\in A$, $i\ge 1,$ that is $m_i\in Hom_k^G(A\otimes A, A),$ for all $i\ge 1.$
\end{enumerate}
\end{defn}
\section{Equivariant Deformation complex of  an associative algebra morphism}\label{rbsec3}
In this section, we introduce equivariant deformation complex of an  associative algebra morphism.  In the subsequent sections we show that second and third cohomology of this complex controls deformation.
\begin{defn}
  Let $G$ be a finite group, A and B be associative algebras with actions $(G,A)$ and $(G,B)$ respectively.  A G-equivariant associative algebra morphism $\phi:A\to B$  is defined to be an  associative algebra morphism such that $\phi(ga)=g\phi(a)$, for all $a\in A$, $g\in G.$
\end{defn}
\begin{defn}
    Let  $\phi:A\to B$  be  a G-equivariant associative algebra morphism.
     \begin{enumerate}
       \item  An equivariant left (or right)  module over $\phi$ is defined to be a G-equivariant k-linear map $T:M\to N$ which is a left (or right) module over $\phi$.
       \item  An equivariant bimodule over $\phi$ is defined to be a G-equivariant k-linear map $T:M\to N$ which is a bimodule over $\phi$. In particular, $\phi$ is an equivariant bimodule over itself.
     \end{enumerate}

\end{defn}
\begin{defn}
Let $\phi:A\to B$ be a G-equivariant associative algebra morphism and $T:M\to N$ be an equivariant bimodule over $\phi$. We define
$$C_G^n(\phi;T)=C_G^n(A;M)\oplus C_G^n(B;N)\oplus C_G^{n-1}(A;N),$$ for all $n\in \mathbb{N}$ and $C_G^0(\phi;T)=0$. Also, we define $d^n:C_G^n(\phi;T)\to C_G^{n+1}(\phi;T)$ by $$d^n(u,v,w)=(\delta^n u, \delta^n v, Tu-v\phi -\delta^{n-1}w ),$$ for all $(u,v,w)\in C_G^n(\phi;T).$ Here the $\delta^n$'s denote coboundaries of the cochain complexes $C_G^{\ast}(A;M)$,  $C_G^{\ast}(B;N)$ and $C_G^{\ast}(A;N)$, the  $Tu$ denotes the composition $T\circ u$ of $T$ and $u$ and the map $v\phi:A^{\otimes n}\to N$ is defined by $v\phi(x_1, x_2,\cdots, x_n)=v(\phi(x_1), \phi(x_2), \cdots, \phi (x_n)).$ From Lemma \ref{rb1}, one can easily verify that $d^n$ is well defined.
\end{defn}
\begin{prop}
  $(C_G^{\ast}(\phi;T), d)$ is a cochain complex.
\end{prop}
\begin{proof}
  We have \begin{eqnarray*}
           d^{n+1}d^n(u,v,w)   &=& d^{n+1}(\delta^n u, \delta^n v, Tu-v\phi -\delta^{n-1}w ) \\
             &=& (\delta^{n+1}\delta^n u, \delta^{n+1}\delta^n v, T(\delta^n u)- (\delta^n v)\phi- \delta^n( Tu-v\phi-\delta^{n-1}w))
          \end{eqnarray*} One can easily see that $\delta^n( Tu-v\phi)=T(\delta^n u)- (\delta^n v)\phi$. So, since $\delta^{n+1}\delta^n u= \delta^{n+1}\delta^n v=\delta^{n+1}\delta^n w=0$, we have $d^{n+1}d^n=0$. Hence we conclude the result.
\end{proof}

We call the cochain complex $(C_G^{\ast}(\phi,\phi),d)$  as equivariant deformation complex of $\phi,$ and the corresponding cohomology as equivariant deformation cohomology of $\phi$. We denote the equivariant deformation cohomology by $H_G^n(\phi,\phi)$, that is $H_G^n(\phi,\phi)=H^n(C_G^{\ast}(\phi,\phi),d)$. Next proposition relates $H_G^{\ast}(\phi,\phi)$ to $H_G^{\ast}(A,A)$,  $H_G^{\ast}(B,B)$ and $H_G^{\ast}(A,B)$.

\begin{prop}\label{rb-99}
  If $H_G^n(A,A)=0$,  $H_G^n(B,B)=0$ and $H_G^{n-1}(A,B)=0$, then $H_G^{n}(\phi,\phi)=0$.
\end{prop}
\begin{proof}
  Let $(u,v,w)\in C_G^n(\phi,\phi)$ be a cocycle, that is $d^n(u,v,w)=(\delta^n u, \delta^n v, \phi u-v\phi -\delta^{n-1}w )=0$. This implies that $\delta^n u=0$, $\delta^n v=0$, $\phi u-v\phi -\delta^{n-1}w =0$. $H_G^n(A,A)=0 \Rightarrow u=\delta^{n-1}u_1$ and  $H_G^n(B,B)=0\Rightarrow \delta^{n-1}v_1=v$, for some $u_1\in C_G^{n-1}(\phi,\phi)$ and $v_1\in C_G^{n-1}(\phi,\phi)$. So $0=\phi u-v\phi -\delta^{n-1}w=\phi(\delta^{n-1}u_1)-(\delta^{n-1}v_1)\phi-\delta^{n-1}w=\delta^{n-1}(\phi u_1)-\delta^{n-1}(v_1\phi)-\delta^{n-1}w=\delta^{n-1}(\phi u_1-v_1\phi-w)$. So $\phi u_1-v_1\phi-w\in C_G^{n-1}(A,B)$ is a cocycle. Now,  $H_G^{n-1}(A,B)=0\Rightarrow \phi u_1-v_1\phi-w=\delta^{n-2}w_1\Rightarrow \phi u_1-v_1\phi-\delta^{n-2}w_1=w$. Thus $(u,v,w)=(\delta^{n-1} u_1, \delta^{n-1 }v_1, \phi u_1-v_1\phi -\delta^{n-2}w_1 )=d^{n-1}(u_1,v_1,w_1)$, for some $(u_1,v_1,w_1)\in C_G^{n-1}(\phi,\phi).$  Thus every cocycle in $C_G^n(\phi,\phi)$ is a coboundary. Hence we conclude that $H_G^{n}(\phi,\phi)=0$.
\end{proof}
\section{ Equivariant deformation of  an associative algebra morphism}\label{rbsec4}
\begin{defn}\label{rb2}

Let A and B be associative k-algebras with actions $(G,A)$ and $(G,B)$ respectively. An equivariant deformation of a G-equivariant associative algebra morphism $\phi:A\to B$ is a triple $(\mu_t, \nu_t, \phi_t)$, in which:

\begin{enumerate}
  \item  $\mu_t=\sum_{i=0}^{\infty}\mu_i t^i$ is an  equivariant formal one-parameter deformation for $A$.
  \item  $\nu_t=\sum_{i=0}^{\infty}\nu_it^i$ is an  equivariant formal one-parameter deformations for  $B$.
  \item  $\phi_t:A_t\to B_t$  is a G-equivariant associative  algebra morphism of the form  $\phi_t=\sum_{i=0}^{\infty}\phi_it^i$, where $\phi_i:A\to B$ is    G-equivariant associative algebra morphisms for all $i\ge 0$ and $\phi_0=\phi$.
\end{enumerate}

Therefore a triple $(\mu_t, \nu_t, \phi_t)$, as given above, is an equivariant deformation of $\phi$ provided following properties are satisfied.
\begin{itemize}
  \item[(i)] $\mu_t(\mu_t(a,b),c)=\mu_t(a,\mu_t(b,c))$, for all $a,b,c\in A;$
  \item[(ii)]  $\mu_i(ga,gb)=g\mu_i(a,b)$, for all $a,b\in A$ and $g\in G;$
  \item[(iii)] $\nu_t(\nu_t(a,b),c)=\nu_t(a,\nu_t(b,c))$, for all $a,b,c\in B;$
  \item[(iv)]  $\nu_i(ga,gb)=g\nu_i(a,b)$, for all $a,b\in B$ and $g\in G;$
  \item[(v)] $\phi_t(\mu_t(a,b))=\nu_t(\phi_t(a),\phi_t(b)), $  for all $a,b\in A;$
  \item[(vi)] $\phi_i(ga)=g\phi_i(a), $ for all $a\in A$ and $g\in G.$
\end{itemize}
The conditions $(i)$, $(iii)$ and $(v)$ are equivalent to following conditions respectively.
\begin{equation}\label{rbeqn1}
  \sum_{i+j=r}\mu_i(\mu_j(a,b),c)=\sum_{i+j=r}\mu_i(a,\mu_j(b,c)), \;\text{for all} a,b,c\in A,\; r\ge 0.
  \end{equation}
  \begin{equation}\label{rbeqn2}
    \sum_{i+j=r}\nu_i(\nu_j(a,b),c)=\sum_{i+j=r}\nu_i(a,\nu_j(b,c)), \;\text{for all} a,b,c\in B, \;r\ge 0.
   \end{equation}
   \begin{equation}\label{rbeqn3}
      \sum_{i+j=r}\phi_i(\mu_j(a,b))=\sum_{i+j+k=r}\nu_i(\phi_j(a),\phi_k(b)); \;\text{for all} a,b\in A,\;  r\ge 0.
\end{equation}

\end{defn}
 Now we define equivariant deformations  of finite order.

 \begin{defn}\label{rb3}
Let A and B be associative k-algebras with actions $(G,A)$ and $(G,B)$ respectively. An equivariant deformation of order n  of a G-equivariant associative algebra morphism $\phi:A\to B$ is a triple $(\mu_t, \nu_t, \phi_t)$, in which:

\begin{enumerate}
  \item  $\mu_t=\sum_{i=0}^{n}\mu_i t^i$ is an  equivariant formal one-parameter deformation of order n for $A$.
  \item  $\nu_t=\sum_{i=0}^{n}\nu_it^i$ is an  equivariant formal one-parameter deformations of order n for  $B$.
  \item  $\phi_t:A_t\to B_t$  is a G-equivariant associative  algebra morphism of the form  $\phi_t=\sum_{i=0}^{n}\phi_it^i$, where $\phi_i:A\to B$ is    G-equivariant associative algebra morphisms for all $i\ge 0$ and $\phi_0=\phi$.
\end{enumerate}

\end{defn}
\begin{rem}\label{rbrem1}
  \begin{itemize}
    \item For $r=0$, conditions \ref{rbeqn1}, \ref{rbeqn2} and \ref{rbeqn3} are equivalent to the fact that A and B are associative algebras and $\phi$ is an associative algebra morphism, respectively.
    \item For $r=1$, \ref{rbeqn1}, \ref{rbeqn2} and \ref{rbeqn3} are equivalent to $\delta^2\mu_1=0,$ $\delta^2\nu_1=0$ and $\phi\mu_1-\nu_1\phi-\delta^1\phi_1=0$.
          Thus for $r=1$,  \ref{rbeqn1}, \ref{rbeqn2} and \ref{rbeqn3} are equivalent to saying that $(\mu_1,\nu_1,\phi_1)\in C^2_G(\phi,\phi)$ is a cocycle. In general, for $r\ge 2$, $(\mu_r,\nu_r,\phi_r)$ is just a 2-cochain in $C^2_G(\phi,\phi).$
  \end{itemize}
\end{rem}
\begin{defn}
  The 2-cochain  $(\mu_1,\nu_1,\phi_1)$ in $C^2_G(\phi,\phi)$ is called infinitesimal of the equivariant  deformation $(\mu_t,\nu_t,\phi_t)$. In general, if $(\mu_i,\nu_i,\phi_i)=0,$ for $1\le i\le n-1$, and $(\mu_n,\nu_n,\phi_n)$ is a nonzero cochain in  $C^2_G(\phi,\phi)$, then $(\mu_n,\nu_n,\phi_n)$ is called n-infinitesimal of the deformation $(\mu_t,\nu_t,\phi_t)$.
\end{defn}
\begin{prop}
  The infinitesimal   $(\mu_1,\nu_1,\phi_1)$ of the equivariant deformation  $(\mu_t,\nu_t,\phi_t)$ is a 2-cocycle in $C^2_G(\phi,\phi).$ In general, n-infinitesimal  $(\mu_n,\nu_n,\phi_n)$ is a cocycle in $C^2_G(\phi,\phi).$
\end{prop}
\begin{proof}
  For n=1, proof is obvious from the Remark \ref{rbrem1}. For $n>1$, proof is similar.
\end{proof}
We can write equations \ref{rbeqn1}, \ref{rbeqn2} and \ref{rbeqn3} for $r=n+1$ using the definition of coboundary $\delta$ as
\begin{equation}\label{rbeqn4}
  \delta^2\mu_{n+1}(a,b,c)=\sum_{\substack{ i+j=n+1\\i,j>0}}\mu_i(\mu_j(a,b),c)-\mu_i(a,\mu_j(b,c)), \;\text{for all} a,b,c\in A.
  \end{equation}
  \begin{equation}\label{rbeqn5}
    \delta^2\nu_{n+1}(a,b,c)=\sum_{\substack{ i+j=n+1\\i,j>0}}\nu_i(\nu_j(a,b),c)-\nu_i(a,\nu_j(b,c)), \;\text{for all} a,b,c\in B.
   \end{equation}
   \begin{eqnarray}\label{rbeqn6}
     \phi(\mu_{n+1}(a,b))&-&\nu_{n+1}(\phi(a),\phi(b))-\delta^1\phi_{n+1}(a,b) \notag\\
      &=& \sum^{\prime}\nu_i(\phi_j(a),\phi_k(b)) -\sum_{\substack{ i+j=n+1\\i,j>0}}\phi_i(\mu_j(a,b)) ,
   \end{eqnarray}
 for all  $a,b\in A$,  where,
  \begin{eqnarray}
           \sum^{\prime} &=& \sum_{\substack{ i+j=n+1\\i,j>0\\k=0}}+\sum_{\substack{ j+k=n+1\\j,k>0\\i=0}}  +\sum_{\substack{ k+i=n+1\\k,i>0\\j=0}}+\sum_{\substack{ i+j+k=n+1\\i,j,k>0}}.
         \end{eqnarray}
By using   equations \ref{rbeqn4}, \ref{rbeqn5} and \ref{rbeqn6}   we have
\begin{align}
&d^2(\mu_{n+1},\nu_{n+1},\phi_{n+1})(a,b,c,x,y,z,p,q) \notag\\
&= ( \sum_{\substack{ i+j=n+1\\i,j>0}}\mu_i(\mu_j(a,b),c)-\mu_i(a,\mu_j(b,c)),
    \sum_{\substack{ i+j=n+1\\i,j>0}}\nu_i(\nu_j(x,y),z)-\nu_i(x,\nu_j(y,z)),\notag\\
   & \hspace{3cm}\sum^{\prime}\nu_i(\phi_j(p),\phi_k(q))-\sum_{\substack{ i+j=n+1\\i,j>0}}\phi_i(\mu_j(p,q))),
   \end{align}
for all $a,b,c,p,q\in A$ and  $x,y,z\in B$.\\
Define a 3-cochain $F_{n+1}$ by
\begin{align}
&F_{n+1}(a,b,c,x,y,z,p,q)\notag\\
&= ( \sum_{\substack{ i+j=n+1\\i,j>0}}\mu_i(\mu_j(a,b),c)-\mu_i(a,\mu_j(b,c)),
\sum_{\substack{ i+j=n+1\\i,j>0}}\nu_i(\nu_j(x,y),z)-\nu_i(x,\nu_j(y,z)),\notag\\
   & \hspace{2cm} \sum^{\prime}\nu_i(\phi_j(p),\phi_k(q))-\sum_{\substack{ i+j=n+1\\i,j>0}}\phi_i(\mu_j(p,q)))
\end{align}

\begin{lem}
  The 3-cochain $F_{n+1}$ is invariant, that is $F_{n+1}\in C_G^3(\phi,\phi).$
\end{lem}
\begin{proof}
 To prove that $F_{n+1}$ is invariant we show that  $$F_{n+1}(ga,gb,gc,gx,gy,gz,gp,gq)=gF_{n+1}(a,b,c,x,y,z,p,q), $$  for all $a,b,c,p,q\in A$  and $x,y,z\in B.$ From definition \ref{rb2}, we have $$\mu_i(ga,gb)=g\mu_i(a,b),\;\;\nu_i(gx,gy)=g\nu_i(x,y),\;\; \phi_i(ga)=g\phi_i(a),$$ for all $a,b\in A$ and $x,y\in B.$
 So, we have, for all $a,b,c,p,q\in A$ and $x,y,z\in B,$
 {\scriptsize \begin{align*}
   &F_{n+1}(ga,gb,gc,gx,gy,gz,gp,gq)\\
   &= ( \sum_{\substack{ i+j=n+1\\i,j>0}}\mu_i(\mu_j(ga,gb),gc)-\mu_i(ga,\mu_j(gb,gc)),
    \sum_{\substack{ i+j=n+1\\i,j>0}}\nu_i(\nu_j(gx,gy),gz)-\nu_i(gx,\nu_j(gy,gz)),\\
   & \hspace{2cm} \sum^{\prime}\nu_i(\phi_j(gp),\phi_k(gq))-\sum_{\substack{ i+j=n+1\\i,j>0}}\phi_i(\mu_j(gp,gq)))\\
   &= ( \sum_{\substack{ i+j=n+1\\i,j>0}}\mu_i(g\mu_j(a,b),gc)-\mu_i(ga,g\mu_j(b,c)),
    \sum_{\substack{ i+j=n+1\\i,j>0}}\nu_i(g\nu_j(x,y),gz)-\nu_i(gx,g\nu_j(y,z)),\\
   &\hspace{2cm} \sum^{\prime}\nu_i(g\phi_j(p),g\phi_k(q))-\sum_{\substack{ i+j=n+1\\i,j>0}}\phi_i(g\mu_j(p,q)))\\
   &= ( \sum_{\substack{ i+j=n+1\\i,j>0}}g\mu_i(\mu_j(a,b),c)-g\mu_i(a,\mu_j(b,c)),
    \sum_{\substack{ i+j=n+1\\i,j>0}}g\nu_i(\nu_j(x,y),z)-g\nu_i(x,\nu_j(y,z)),\\
   & \hspace{2cm} \sum^{\prime}g\nu_i(\phi_j(p),\phi_k(q))-\sum_{\substack{ i+j=n+1\\i,j>0}}g\phi_i(\mu_j(p,q)))\\
   &=gF_{n+1}(a,b,c,x,y,z,p,q).
 \end{align*}}
 So we conclude that $F_{n+1}\in C_G^n(\phi,\phi).$
\end{proof}
\begin{defn}
  The 3-cochain $F_{n+1}\in C_G^n(\phi,\phi)$ is called $(n+1)th$ obstruction cochain for extending the given equivariant deformation of order n to an equivariant deformation of $\phi$ of order $(n+1)$. Now onwards we denote $F_{n+1}$ by $Ob_{n+1}(\phi_t)$
\end{defn}
We have the following result.
\begin{thm}
  The (n+1)th obstruction cochain $Ob_{n+1}(\phi_t)$ is a 3-cocycle.
\end{thm}
\begin{proof}
  We have,
  $$d^3 Ob_{n+1}=(\delta^3(O_1),\delta^3(O_2),\phi O_1-O_2\phi-\delta^2(O_3)),$$
  where $O_1$, $O_2$ and $O_3$ are given by  $$O_1(a,b,c)= \sum_{\substack{ i+j=n+1\\i,j>0}}\left\{\mu_i(\mu_j(a,b),c)-\mu_i(a,\mu_j(b,c))\right\},$$
  $$O_2(x,y,z)= \sum_{\substack{ i+j=n+1\\i,j>0}}\left \{\nu_i(\nu_j(x,y),z)-\nu_i(x,\nu_j(y,z))\right\},$$
  $$O_3(p,q)= \sum^{\prime}\nu_i(\phi_j(p),\phi_k(q)) -\sum_{\substack{ i+j=n+1\\i,j>0}}\phi_i(\mu_j(p,q)).$$
  From \cite{gmrb}, we have $\delta^3(O_1)=0$,  $\delta^3(O_2)=0$. So, to prove that $d^3Ob_{n+1}=0$, it remains to show that $\phi O_1-O_2\phi-\delta^2(O_3)=0.$
  To prove that $\phi O_1-O_2\phi-\delta^2(O_3)=0$ we use similar ideas as have been used in \cite{AM} and \cite{DY}.
  We have,
  \begin{align}\label{rbeqn7}
   (\phi O_1 -O_2\phi )(x,y,z)&=\sum_{\substack{ i+j=n+1\\i,j>0}}\phi \mu_i(\mu_j(x,y),z)\notag -\sum_{\substack{ i+j=n+1\\i,j>0}}\phi \mu_i(x,\mu_j(y,z))\nonumber\\
    & -\sum_{\substack{ i+j=n+1\\i,j>0}}\nu_i(\nu_j(\phi x,\phi y),\phi z) +\sum_{\substack{ i+j=n+1\\i,j>0}}\nu_i(\phi x,\nu_j(\phi y,\phi z))
 \end{align}
  and \begin{eqnarray}\label{rbeqn8}
        \delta^2(O_3)(x,y,z) &=& \sum^{\prime}\nu_0(\phi (x),\nu_i(\phi_j(y),\phi_k(z))) -\sum_{\substack{ i+j=n+1\\i,j>0}}\nu_0(\phi(x),\phi_i(\mu_j(y,z))) \notag\\
         && -\sum^{\prime}\nu_i(\phi_j(\mu_0(x,y)),\phi_k(z)) +\sum_{\substack{ i+j=n+1\\i,j>0}}\phi_i(\mu_j(\mu_0(x,y),z)) \notag\\
         && +\sum^{\prime}\nu_i(\phi_j(x),\phi_k(\mu_0(y,z))) -\sum_{\substack{ i+j=n+1\\i,j>0}}\phi_i(\mu_j(x,\mu_0(y,z))) \notag\\
         &&-\sum^{\prime}\nu_0(\nu_i(\phi_j(x),\phi_k(y)),\phi(z))\notag \\
         && +\sum_{\substack{ i+j=n+1\\i,j>0}}\nu_0(\phi_i(\mu_j(x,y)),\phi(z))
\end{eqnarray}
From \ref{rbeqn3}, we have
\begin{eqnarray}\label{rbeqn9}
   \phi_j\mu_0(x,y)&=& \sum_{\substack{ \alpha+\beta+\gamma=j\\\alpha,\beta,\gamma\ge 0}}\nu_{\alpha}(\phi_{\beta}(x),\phi_{\gamma}(y))-\sum_{\substack{ p+q=j\\1\le q\le j}}\phi_p\mu_q(x,y)
\end{eqnarray}
Substituting expression for $\phi_j\mu_0$ from \ref{rbeqn9}, in the third sum on the right hand side of \ref{rbeqn8} we can rewrite it as
\begin{eqnarray}\label{rbeqn10}
  -\sum^{\prime}\nu_i(\phi_j(\mu_0(x,y)),\phi_k(z))   &=&-\sum^{\prime}_{\substack{ \alpha+\beta+\gamma=j\\\alpha,\beta,\gamma\ge 0}}\nu_i(\nu_{\alpha}(\phi_{\beta}(x),\phi_{\gamma}(y)),\phi_k(z))\nonumber\\
  &&+\sum^{\prime}_{\substack{ p+q=j\\1\le q\le j}}\nu_i(\phi_p\mu_q(x,y),\phi_k(z))
\end{eqnarray}
Here the first sum of \ref{rbeqn10} is given by
\begin{eqnarray}\label{rbeqn11}
           \sum^{\prime}_{\substack{ \alpha+\beta+\gamma=j\\\alpha,\beta,\gamma\ge 0}} &=& \sum_{\substack{ i+\alpha+\beta+\gamma=n+1\\i,(\alpha+\beta+\gamma)>0\\k=0}}+\sum_{\substack{ \alpha+\beta+\gamma+k=n+1\\(\alpha+\beta+\gamma),k>0\\i=0}}  +\sum_{\substack{ k+i=n+1\\k,i>0\\\alpha=\beta=\gamma=0}}+\sum_{\substack{ i+\alpha+\beta+\gamma+k=n+1\\i,(\alpha+\beta+\gamma),k>0}},
         \end{eqnarray}
  the second sum   of \ref{rbeqn10} is given by
 \begin{eqnarray}\label{rbeqn12}
    \sum^{\prime}_{\substack{ p+q=j\\1\le q\le j}} &=& \sum_{\substack{ i+p+q=n+1\\i,q>0,p\ge 0\\k=0}}+\sum_{\substack{ p+q+k=n+1\\q,k>0, p\ge 0\\i=0}} +\sum_{\substack{ i+j+k=n+1\\i,q,k>0, p\ge 0}}
    \end{eqnarray}
         The first sum of \ref{rbeqn9} splits into four sums. The  first one of these four sums splits  as
         \begin{align}\label{rbeqn13}
         & - \sum_{\substack{ i+\alpha+\beta+\gamma=n+1\\i,(\alpha+\beta+\gamma)>0\\k=0}}\nu_i(\nu_{\alpha}(\phi_{\beta}(x),\phi_{\gamma}(y)),\phi_k(z)) \nonumber\\
          &= -\sum_{\substack{i+\alpha=n+1\\i,\alpha>0}}\nu_i(\nu_{\alpha}(\phi(x),\phi(y)),\phi(z)))
             -\sum_{\substack{i+\alpha+\beta+\gamma=n+1\\i,(\beta+\gamma)>0\\\alpha,\beta,\gamma\ge 0}}\nu_i(\nu_{\alpha}(\phi_{\beta}(x),\phi_{\gamma}(y)),\phi(z)).
         \end{align}
         The fist sum on the r.h.s. of \ref{rbeqn13} appears as third sum on the  r.h.s. of \ref{rbeqn7}.
         By applying a similar arguement to the fifth sum on the r.h.s. of \ref{rbeqn8}, using \ref{rbeqn3} on $\phi_k\mu_0(y,z)$, one can rewrite it as
         \begin{eqnarray}\label{rbeqn14}
  \sum^{\prime}\nu_i(\phi_j(x),\phi_k\mu_0(y,z))   &=&\sum^{\prime}_{\substack{ \alpha+\beta+\gamma=j\\\alpha,\beta,\gamma\ge 0}}\nu_i(\phi_j(x),\nu_{\alpha}(\phi_{\beta}(y),\phi_{\gamma}(z)))  \nonumber\\
  &&-\sum^{\prime}_{\substack{ p+q=k\\1\le q\le j}}\nu_i(\phi_j(x),\phi_p\mu_q(y,z))
\end{eqnarray}
As above first sum on r.h.s. of  \ref{rbeqn14} is a sum of  four sums, similar to \ref{rbeqn11} except that the  roles of j and k are interchanged. One of these four terms splits as
\begin{align}\label{rbeqn15}
& \sum_{\substack{ i+\alpha+\beta+\gamma=n+1\\i,(\alpha+\beta+\gamma)>0\\k=0}}\nu_i(\phi_j(x),\nu_{\alpha}(\phi_{\beta}(y),\phi_{\gamma}(z))) \nonumber\\
&= \sum_{\substack{i+\alpha=n+1\\i,\alpha>0}}\nu_i(\phi(x),\nu_{\alpha}(\phi(y),\phi(z)))
 +\sum_{\substack{i+\alpha+\beta+\gamma=n+1\\i,(\beta+\gamma)>0\\\alpha,\beta,\gamma\ge 0}}\nu_i(\phi(x),\nu_{\alpha}(\phi_{\beta}(y),\phi_{\gamma}(z))).
\end{align}
The fist sum on the r.h.s. of \ref{rbeqn15} appears as fourth sum on the  r.h.s. of \ref{rbeqn7}.
In the fourth sum on the r.h.s. of \ref{rbeqn8}, we use \ref{rbeqn1} to substitute $\mu_j(\mu_0(x,y),z)$ to obtain
\begin{eqnarray}\label{rbeqn16}
  \sum_{\substack{i+j=n+1\\i,j>0}}\phi_i\mu_j(\mu_0(x,y),z) &=&\sum_{\substack{i+j=n+1\\i,j>0}}\phi_i\mu_j(x,\mu_0(y,z)+\sum_{\substack{i+j+k=n+1\\i,k>0, j\ge 0}}\phi_i\mu_j(x,\mu_k(y, z))\nonumber\\
  &&-\sum_{\substack{i+j+k=n+1\\i,k>0, j\ge 0}}\phi_i\mu_j(\mu_k(x, y), z).
\end{eqnarray}
Fist sum on the r.h.s. of \ref{rbeqn16} cancels with the sixth sum on the r.h.s. of \ref{rbeqn8}.
Second sum on the r.h.s. 0f \ref{rbeqn16} splits as
\begin{eqnarray}\label{rbeqn17}
  \sum_{\substack{i+j+k=n+1\\i,k>0, j\ge 0}}\phi_i\mu_j(x,\mu_k(y, z)) &=& \sum_{k=1}^n\sum_{\substack{i+j+k=n+1\\i,j\ge 0}}\phi_i\mu_j(x,\mu_k(y,z) \nonumber\\
   && -\phi\sum_{\substack{j+k=n+1\\j,k>0}}\mu_j(x,mu_k(y,z))
\end{eqnarray}
Second sum on the r.h.s. of \ref{rbeqn17} appears as second sum on the r.h.s. of \ref{rbeqn7}.
Also, by using \ref{rbeqn3} first sum on the r.h.s. of \ref{rbeqn17} splits as

\begin{eqnarray}\label{rbeqn18}
  \sum_{k=1}^n\sum_{\substack{i+j+k=n+1\\i,j\ge 0}}\phi_i\mu_j(x,\mu_k(y,z)  &=& \sum_{k=1}^n\sum_{\substack{\alpha+\beta+\gamma+k=n+1\\\alpha,\beta,\gamma\ge 0}}\nu_{\alpha}(\phi_{\beta},\phi_{\gamma}\mu_k(y,z)) \nonumber \\
   &=& \sum_{\substack{i+j=n+1\\i,j> 0}}\nu_0(\phi(x), \phi_i\mu_j(y,z))\nonumber\\
   &&+\sum^{\prime}_{\substack{p+q=k\\i\le q\le k}}\nu_i(\phi_j(x), \phi_p\mu_q(y,z))
\end{eqnarray}
In the last line the two terms cancel with second terms on the r.h.s of  \ref{rbeqn8} and \ref{rbeqn14}, respectively.
The third term on the r.h.s. of \ref{rbeqn17} splits as
\begin{eqnarray}\label{rbeqn19}
 -\sum_{\substack{i+j+k=n+1\\i,k>0, j\ge 0}}\phi_i\mu_j(\mu_k(x, y), z)
 &=& -\sum_{\substack{i+j=n+1\\i,j>0}}\nu_0(\phi_i\mu_j(x,y), \phi (z))\nonumber\\
 &&- \sum^{\prime}_{\substack{p+q=j\\1\le q\le j}}\nu_i(\phi_p\mu_q(x,y), \phi_k (z))\nonumber\\
   & & +\sum_{\substack{j+k=n+1\\j,k>0}}\phi \mu_j(\mu_k(x,y), z).
\end{eqnarray}
On the r.h.s. of \ref{rbeqn19}, last term cancels with first sum on the r.h.s. of \ref{rbeqn7}, first sum cancels with the last sum on the r.h.s. of \ref{rbeqn8} and second term cancels with the second sum on the r.h.s. of \ref{rbeqn10}. From our previous arquements we have,
\begin{eqnarray}\label{rbeqn20}
  & &\phi O_1-O_2\phi-\delta^2(O_3)(x,y,z)\nonumber\\
   &=& \sum^{\prime}\left\{\nu_0(\phi(x),\nu_i(\phi_j(y),\phi_k(z)))-\nu_0(\nu_i(\phi_j(x),\phi_k(y),\phi(z)))\right\}\nonumber\\
  & &+\sum_{\substack{i+\alpha+\beta+\gamma=n+1\\i,(\beta+\gamma)>0\\\alpha, \beta,\gamma\ge 0}}\left\{\nu_i(\phi(x),\nu_{\alpha}(\phi_{\beta}(y),\phi_{\gamma}(z))-\nu_i(\nu_{\alpha}(\phi_{\beta}(x),\phi_{\gamma}(y), \phi(z))\right \}\nonumber\\
  & & -\left\{\sum_{\substack{ \alpha+\beta+\gamma+k=n+1\\(\alpha+\beta+\gamma),k>0\\i=0}}  +\sum_{\substack{ k+i=n+1\\k,i>0\\\alpha=\beta=\gamma=0}}+\sum_{\substack{ i+\alpha+\beta+\gamma+k=n+1\\i,(\alpha+\beta+\gamma),k>0}}\right\}\nu_i(\nu_{\alpha}(\phi_{\beta}(x),\phi_{\gamma}(y)),\phi_k(z))\nonumber\\
    & & + \left\{\sum_{\substack{ \alpha+\beta+\gamma+j=n+1\\(\alpha+\beta+\gamma),j>0\\i=0}}  +\sum_{\substack{ i+j=n+1\\i,j>0\\\alpha=\beta=\gamma=0}}+\sum_{\substack{ i+j+\alpha+\beta+\gamma=n+1\\i,j,(\alpha+\beta+\gamma)>0}}\right\}\nu_i(\phi_j(x),\nu_{\alpha}(\phi_{\beta}(y),\phi_{\gamma}(z)))\nonumber \\
    \end{eqnarray}
  We can write above equation more compactly as
  \begin{eqnarray}\label{rbeqn21}
   & &\phi O_1-O_2\phi-\delta^2(O_3)(x,y,z)\nonumber\\
   &= &\overline{\sum}\left\{\nu_i(\phi_{\alpha}(x),\nu_j(\phi_{\beta}(y),\phi_{\gamma}(z))) -\nu_i(\nu_j(\phi_{\alpha}(x),\phi_{\beta}(y)),\phi_{\gamma}(z))\right\},
  \end{eqnarray}
  where
  \begin{eqnarray}
    \overline{\sum} &=& \sum_{\substack{i+j+\alpha+\beta+\gamma=n+1\\1\le \alpha+\beta+\gamma \le n\\i,j,\alpha,\beta,\gamma\ge 0}} + \sum_{\substack{\alpha+\beta=n+1\\\alpha,\beta>0\\i,j,\gamma =0}}  + \sum_{\substack{\beta+\gamma=n+1\\\beta,\gamma>0\\i,j,\alpha=0}}\nonumber \\
     & &+ \sum_{\substack{\alpha+\gamma=n+1\\\alpha,\gamma>0\\i,j,\beta =0}} + \sum_{\substack{\alpha+\beta+\gamma=n+1\\\alpha,\beta,\gamma>0\\i,j=0}}.
  \end{eqnarray}
  It follows  from \ref{rbeqn21} and \ref{rbeqn2} that the sum on the r.h.s. of \ref{rbeqn21} is 0, and hence $\phi O_1-O_2\phi-\delta^2(O_3)=0$.
  This finishes the proof of the theorem.
\end{proof}
\begin{thm}
Let $(\mu_t,\nu_t,\phi_t)$ be an   equivariant deformation of $\phi$ of order n. Then $(\mu_t,\nu_t,\phi_t)$ extends to an equivariant deformation of order $n+1$ if and only if cohomology class of $(n+1)$th obstruction $Ob_{n+1}(\phi_t)$  vanishes.
\end{thm}
\begin{proof}
  Suppose that an equivariant deformation $(\mu_t,\nu_t,\phi_t)$ of $\phi$ of order n extends to an equivariant deformation of order $n+1$. This implies that \ref{rbeqn1},\ref{rbeqn2} and \ref{rbeqn3} are satisfied for $r=n+1.$  Observe that this implies $Ob_{n+1}(\phi_t)=d_2(\mu_{n+1},\nu_{n+1},\phi_{n+1})$. So cohomology class of  $Ob_{n+1}(\phi_t)$ vanishes. Conversely, suppose that  cohomology class of  $Ob_{n+1}(\phi_t)$ vanishes, that is $Ob_{n+1}(\phi_t)$ is a coboundary. Let
  $$ Ob_{n+1}(\phi_t)=d_2(\mu_{n+1},\nu_{n+1},\phi_{n+1}),$$
  for some 2-cochain  $(\mu_{n+1},\nu_{n+1},\phi_{n+1})\in C^2_G(\phi,\phi).$ Take
  $$(\tilde{\mu_t},\tilde{\nu_t},\tilde{\phi_t})=(\mu_t+\mu_{n+1}t^{n+1},\nu_t+\nu_{n+1}t^{n+1},\phi_t+\phi_{n+1}t^{n+1})$$.
  Observe that $(\tilde{\mu_t},\tilde{\nu_t},\tilde{\phi_t})$ satisfies  \ref{rbeqn1},\ref{rbeqn2} and \ref{rbeqn3} for $0\le r\le n+1$. So  $(\tilde{\mu_t},\tilde{\nu_t},\tilde{\phi_t})$ is an equivariant extension of $(\mu_t,\nu_t,\phi_t)$ of order $n+1$.

\end{proof}
\begin{cor}
  If $H^3_G(\phi,\phi)=0$, then every 2-cocycle in $C^2_G(\phi,\phi)$ is an infinitesimal of some equivariant deformation of $\phi.$
\end{cor}

\section{Equivalence of equivariant deformations, and rigidity }\label{rbsec5}
Let   $(\mu_t,\nu_t,\phi_t)$  and $(\tilde{\mu_t},\tilde{\nu_t},\tilde{\phi_t})$ be two equivariant deformations of $\phi$. Recall from \cite{gmrb} that an equivariant formal isomorphism between the equivariant deformations $\mu_t$ and $\tilde{\mu_t}$ of an associative algebra A is a $k[[t]]$-linear G-automorphism $\Psi_t:A[[t]]\to A[[t]]$ of the  form  $\Psi_t=\sum_{i\ge 0}\psi_it^i$, where each $\psi_i$ is an equivariant $k$-linear map $A\to A$, $\psi_0(a)=a$, for all $a\in A$ and $\tilde{\mu_t}(\Psi_t(a),\Psi_t(b))=\Psi_t\mu_t(a,b),$ for all $a,b\in A$
\begin{defn}
  An equivariant  formal isomorphism from $(\mu_t,\nu_t,\phi_t)$  to  $(\tilde{\mu_t},\tilde{\nu_t},\tilde{\phi_t})$ is a pair  $(\Psi_t,\Theta_t)$, where $\Psi_t:A[[t]]\to A[[t]]$ and $\Theta_t:B[[t]]\to B[[t]]$ are equivariant formal isomorphisms from $\mu_t$ to $\tilde{\mu_t}$ and $\nu_t$ to $\tilde{\nu_t}$, respectively,  such that $$\tilde{\phi_t}\circ\Psi_t=\Theta_t\circ\phi_t.$$
  Two equivariant deformations $(\mu_t,\nu_t,\phi_t)$  and $(\tilde{\mu_t},\tilde{\nu_t},\tilde{\phi_t})$ are said to be equivalent if there exists an equivariant formal isomorphism  $(\Psi_t,\Theta_t)$ from $(\mu_t,\nu_t,\phi_t)$ to  $(\tilde{\mu_t},\tilde{\nu_t},\tilde{\phi_t})$.
\end{defn}
\begin{defn}
  Any equivariant deformation of $\phi:A\to B$ that is equialent to the deformation $(\mu_0,\nu_0,\phi)$ is said to be a trivial deformation.
\end{defn}
\begin{thm}
  The cohomology class of the infinitesimal of an equivariant  deformation $(\mu_t,\nu_t,\phi_t)$ of $\phi:A\to B$ is determined by the equivalence class of $(\mu_t,\nu_t,\phi_t)$.
\end{thm}
\begin{proof}
  Let  $(\Psi_t,\Theta_t)$ from  $(\mu_t,\nu_t,\phi_t)$ to  $(\tilde{\mu_t},\tilde{\nu_t},\tilde{\phi_t})$ be an equivariant  formal  isomorphism. So, we have  $\tilde{\mu_t}\Psi_t=\Psi_t\circ \mu_t,$ $\tilde{\nu_t}\Theta_t=\Theta_t\circ \nu_t,$ and   $\tilde{\phi_t}\circ\Psi_t=\Theta_t\circ\phi_t.$ This implies that $\mu_1-\tilde{\mu_1}=\delta\psi_1$, $\nu_1-\tilde{\nu_1}=\delta\theta_1$ and $\phi_1-\tilde{\phi_1}=\phi\psi_1-\theta_1\phi$. So we have $d^1(\psi_1,\theta_1,0)=(\mu_1,\nu_1,\phi_1)-(\tilde{\mu_1},\tilde{\nu_1},\tilde{\phi_1}).$ This finishes the proof.
\end{proof}
\begin{defn}
  An equivariant associative algebra morphism $\phi:A\to B$ is said to be rigid if every deformation of $\phi$ is trivial.
\end{defn}
\begin{thm}\label{rb-100}
  A non-trivial equivariant deformation of an associative algebra morphism is equivalent to an equivariant  deformation whose n-infinitesimal is not a coboundary, for some $n\ge 1.$
\end{thm}
\begin{proof}
  Let $(\mu_t,\nu_t,\phi_t)$ be an equivariant deformation of $\phi$ with n-infinitesimal $(\mu_n,\nu_n,\phi_n)$, for some $n\ge 1.$ Assume that there exists a 1-cochain $(\psi,\theta, m)\in C_G^1(\phi,\phi)$ with $d(\psi,\theta, m)=(\mu_n,\nu_n,\phi_n).$ Since $d(\psi,\theta, m)=d(\psi,\theta+\delta m, 0)$, without any loss of generality we may assume $m=0.$ This gives $\mu_n=\delta \psi$, $\nu_n=\delta\theta$, $\phi_n=\phi\psi-\theta\phi$. Take $\Psi_t=Id_A+\psi t^n$, $\Theta_t=Id_B=\theta t^n$. Define $\tilde{\mu_t}=\Psi_t\circ \mu_t\circ \Psi_t^{-1},$ $\tilde{\nu_t}=\Theta_t\circ \nu_t\circ\Theta_t^{-1},$ and   $\tilde{\phi_t}=\Theta_t\circ\phi_t\circ\Psi_t^{-1}$. Clearly, $(\tilde{\mu_t},\tilde{\nu_t},\tilde{\phi_t})$ is an equivariant deformation of $\phi$ and $(\Psi_t,\Theta_t)$ is an   equivariant  formal isomorphism from $(\mu_t,\nu_t,\phi_t)$  to  $(\tilde{\mu_t},\tilde{\nu_t},\tilde{\phi_t})$. For $u,v\in A,$ we have  $\tilde{\mu_t}(\Psi_tu,\Psi_tv)=\Psi_t(\mu_t(u,v)),$ which implies $\tilde{\mu_i}=0,$ for $1\le i\le n.$ For $u,v\in B$, we have $\tilde{\nu_t}(\Theta_tu,\Theta v)=\Theta_t( \nu_t(u,v)),$ which implies $\tilde{\nu_i}=0,$ for $1\le i\le n.$  For $u\in A,$  we have  $\tilde{\phi_t}(\Psi_tu)=\Theta_t(\phi_tu), $ which gives $\phi_i=0,$  for $1\le i\le n.$ So $(\tilde{\mu_t},\tilde{\nu_t},\tilde{\phi_t})$ is equivalent to the given deformation and $(\tilde{\mu_i},\tilde{\nu_i},\tilde{\phi_i})=0,$ for $1\le i\le n.$ We can repeat the arguement to get rid off any infinitesimal that is a coboundary. So the process must stop if the deformation is nontrivial.
\end{proof}
An immediate consequence of the Theorem \ref{rb-100} is following corollary.
\begin{cor}
  If $H_G^2(\phi,\phi)=0,$ then $\phi:A\to B$ is rigid.
\end{cor}
From Proposition \ref{rb-99} and Theorem \ref{rb-100}, we conclude following Corollary.
\begin{cor}
  If  $H_G^2(A,A)=0,$  $H_G^2(B,B)=0,$  and $H_G^1(A,B)=0,$ then  $\phi:A\to B$ is rigid.
\end{cor}

%

\end{document}